\documentclass [leqno, 10pt]{amsart}
\usepackage[T1]{fontenc}
\usepackage[utf8]{inputenc}

\usepackage[all]{xy}
\usepackage{amssymb, amsmath, amsfonts, amsthm, amsbsy, amscd, stmaryrd}
\usepackage[bbgreekl]{mathbbol} 
\usepackage[mathscr]{euscript}
\usepackage{esint}
\usepackage{upgreek}
\usepackage{xypic}

\usepackage{url}
\usepackage[linktocpage]{hyperref}

\newtheorem{theorem}{Theorem}[section]
\newtheorem*{theorem*}{Theorem}{\,}
\newtheorem{corollary}{Corollary}[section]
\newtheorem{lemma}{Lemma}[section]

\newtheorem*{definition*}{Definition}{}

\theoremstyle{definition}
\newtheorem{example}{Example}[section]
\newtheorem{remark}{Remark}[section]

\numberwithin{equation}{section}

\def\cprime{$'$}

\newcommand{\ind}{\Upsilon}
\newcommand{\smc}{\mathscr{H}}

\newcommand{\sad}{\operatorname{s}}

\newcommand{\dl}{\ell}

\newcommand{\rice}{A}

\newcommand{\scurv}{\mathscr{K}}
\newcommand{\hcurv}{\mathscr{H}}

\newcommand{\rf}{\rho}

\newcommand{\W}{\mathscr{W}}
\newcommand{\Z}{\mathscr{Z}}

\newcommand{\symcon}{\mathbb{S}}

\newcommand{\Om}{\Omega}

\newcommand{\blD}{\mathbb{D}}
\newcommand{\sflat}{\curlyvee}

\newcommand{\om}{\omega}

\newcommand{\lf}{\mathsf{L}}

\newcommand{\sR}{\mathscr{R}}

\renewcommand{\part}{\vdash}

\newcommand{\Id}{\text{Id}}

\newcommand{\dum}{\,\cdot\,\,}
\newcommand{\Ga}{\Gamma}

\newcommand{\ric}{\text{Ric}}

\newcommand{\Q}{\mathcal{Q}}

\renewcommand{\j}{\mathsf{i}}

\newcommand{\ep}{\epsilon}

\newcommand{\ext}{\mbox{\Large $\wedge$}}

\newcommand{\eno}{\operatorname{End}}
\newcommand{\sgr}{\operatorname{SGr}}

\newcommand{\si}{\sigma}

\newcommand{\ctm}{T^{\ast}M}

\newcommand{\gr}{\text{Gr}}

\newcommand{\bnabla}{\bar{\nabla}}

\newcommand{\lie}{\mathfrak{L}}

\newcommand{\im}{\textnormal{im\,}}

\renewcommand{\xi}{\frac{\partial}{\partial x^{i}}}

\newcommand{\aff}{\mathbb{aff}}

\newcommand{\lb}{\langle}
\newcommand{\ra}{\rangle}

\newcommand{\ste}{\mathbb{V}}

\newcommand{\al}{\alpha}

\newcommand{\ga}{\gamma}

\newcommand{\spn}{\operatorname{\mathsf{span}}}

\newcommand{\g}{\mathfrak{g}}

\newcommand{\ad}{\operatorname{ad}}

\newcommand{\tensor}{\otimes}

\newcommand{\rea}{\mathbb R}
\newcommand{\com}{\mathbb C}

\newcommand{\tr}{\operatorname{\mathsf{tr}}}

\begin{document}
\title[Symplectic sectional curvature]{Remarks on symplectic sectional curvature}
\author{Daniel J.~F. Fox} 
\address{Departamento de Matemáticas del Área Industrial\\ Escuela Técnica Superior de Ingeniería y Diseño Industrial\\ Universidad Politécnica de Madrid\\Ronda de Valencia 3\\ 28012 Madrid España}
\email{daniel.fox@upm.es}

\begin{abstract}
In \cite{Gelfand-Retakh-Shubin}, I.~M. Gelfand, V. Retakh, and M. Shubin defined the symplectic sectional curvature of a torsion-free connection preserving a symplectic form. The present article defines the corresponding notion of constant symplectic sectional curvature and characterizes this notion in terms of the curvature tensor of the symplectic connection and its covariant derivatives. Some relations between various more general conditions on the symplectic sectional curvature and the geometry of the symplectic connection or that induced on a symplectic submanifold are explored as well.
\end{abstract}

\maketitle

\section{Introduction}
Let $(M, \Om)$ be a connected smooth symplectic manifold of dimension $2n \geq 2$, oriented by the volume form $\Om^{n}$. An affine connection on $M$ is \textit{symplectic} if it is torsion free and $\nabla\Om = 0$. Although the most accessible examples of symplectic connections are Levi-Civita connections of Kähler, pseudo-Kähler, and para-Kähler metrics, even symplectic manifolds that admit no such compatible metric structure admit symplectic connections, for every symplectic manifold admits symplectic connections (see \eqref{gensym}). Traditionally, much of the interest in symplectic connections has focused on their use in schemes of deformation quantization such as that of Fedosov \cite{Fedosov}, but because they exist on any symplectic manifold, it is also interesting to study their geometry in the spirit of classical metric differential geometry.

The basic identities for the curvature of a symplectic connection were probably first obtained by I. Vaisman in \cite{Vaisman}, although there is information in earlier work of other authors, for example A. Lichnerowicz's \cite{Lichnerowicz-propagateurs} and \cite{Lichnerowicz-starproducts}. Among the many available references some basic ones are \cite{Bayen-Flato-Fronsdal-Lichnerowicz-Sternheimer}, \cite{Fedosov}, and \cite{Gelfand-Retakh-Shubin}, in addition to \cite{Vaisman}. The survey \cite{Bieliavsky-Cahen-Gutt-Rawnsley-Schwachhofer} is a good starting point and contains further references.

The relation between the geometry determined by a symplectic connection $\nabla$ and conditions on quantities and tensors constructed from its curvature is incompletely understood. 
For example, in \cite{Gelfand-Retakh-Shubin}, I.~M. Gelfand, V. Retakh, and M. Shubin defined the symplectic sectional curvature of a symplectic connection. Not much has been done with this notion, and its geometric content has been little explored. This note defines the corresponding notion of constant symplectic sectional curvature, characterizes it in terms of the curvature tensor of the symplectic connection and its covariant derivatives, and describes some related constructions. 

The symplectic sectional curvature of a symplectic $2$-plane $L$ is a quadratic form on $L$ rather than a number (see Section \ref{curvaturesection} for the definition), as in the metric setting. Theorem \ref{scurvtheorem} shows that the symplectic sectional curvature is determined entirely by the restriction of the Ricci tensor exactly when the symplectic Weyl tensor vanishes. As a consequence, it is sensible to say that a symplectic connection has constant symplectic sectional curvature if for every symplectic $2$-plane $L$ it equals the restriction of the quadratic form determined by a \textit{parallel} symmetric two tensor.  Corollary \ref{constantcurvcorollary} shows that a symplectic connection has constant symplectic sectional curvature if and only if its symplectic Weyl tensor vanishes and its Ricci tensor is parallel, in which case it is locally symmetric. 
Corollary \ref{ncpccorollary} shows that in this case the Ricci endomorphism (obtained by raising one index of the Ricci tensor using the symplectic form) is parallel, and must be nilpotent of order two, complex, or paracomplex. These conclusions are closely related to and in part can be obtained from those obtained for homogeneous symplectic connections in \cite{Cahen-Gutt-Horowitz-Rawnsley} and for symplectic symmetric spaces (as defined by P. Bieliavsky in \cite{Bieliavsky-thesis}) in \cite{Cahen-Gutt-Rawnsley-symmetric}, as is explained in more detail in Remark \ref{ssremark}.
The relation between the symplectic sectional curvatures and the (para-)holomorphic sectional curvatures of a pseudo-Kähler or para-Kähler structure is also discussed. The precise relation is given in Lemma \ref{grslemma}, that has as a corollary the result of \cite{Gelfand-Retakh-Shubin} that the symplectic sectional curvatures of a Kähler metric cannot have indefinite signature. Corollary \ref{chcorollary} shows that $\nabla$ has constant nonzero symplectic sectional curvature if and only if it is a complex projective or complex hyperbolic space form. 

The study of pseudo-Riemannian manifolds has in common with the study of symplectic connections that the sectional curvature is defined only for nondegenerate subspaces. A Lorentzian manifold of dimension at least $3$ has constant sectional curvature if and only if it has vanishing null sectional curvature (the definitions are recalled in Section \ref{isotropicsection}). In Section \ref{isotropicsection} there is defined for a symplectic connection a notion of vanishing isotropic sectional curvature analogous to the notion of vanishing null sectional curvature, and Lemma \ref{nulllemma} shows that on a symplectic manifold of dimension at least $4$ a symplectic connection has constant symplectic sectional curvature if and only if it has vanishing isotropic sectional curvature.

In Section \ref{submanifoldsection} there is computed the relation between the symplectic sectional curvature of a symplectic connection and the symplectic sectional curvature of the symplectic connection induced on a symplectic submanifold. The absence of the positivity provided by a Riemannian metric means that the formulas obtained are not as obviously useful as their metric counterparts, for example, in the Kähler setting. There is no obvious analogue of the mean curvature vector field, nor any other tensor associated with a submanifold that is linear in the second fundamental form, other than the form itself. However, the formula \eqref{scurvhereditary} for the symplectic sectional curvature of a symplectic submanifold helps identify two tensors (defined in \eqref{smcdefined}) constructed from expressions quadratic in the second fundamental form. Precisely, if $\Pi_{ij}\,^{A}$ is the second fundamental form, these tensors are the pure trace part $\smc_{ij}$ and the trace-free part $\smc_{ijkl}$ of $\Pi_{ik}\,^{A}\Pi_{jl}\,^{B}\Om_{AB} + \Pi_{il}\,^{A}\Pi_{jk}\,^{B}\Om_{AB}$ (see Section \ref{submanifoldsection} for the notational conventions used). It would be interesting to understand the geometric meaning of the vanishing of $\smc_{ij}$ or $\smc_{ijkl}$.

In the metric setting the distance and volume determined by the metric are linked to the curvature via the formulas for their first and second variations. In particular, Jacobi fields are a basic tool. In the symplectic setting it is difficult to link directly the curvature with genuinely geometric quantities, but there is a more serious difficulty that is readily apparent when one tries to mimic metric arguments using Jacobi fields. The symplectic curvature tensor is symmetric rather than antisymmetric in two of its arguments, and this means that quantities that vanish in the metric setting do not vanish in the symplectic setting. This apparently silly technical problem makes it more difficult to control Jacobi fields. What little can be gleaned easily is recounted in Section \ref{npcsection}. In Section \ref{npcsection} it is proved that any pair of points conjugate with respect to some geodesic of a symplectic connection with nonpositive symplectic sectional curvature correspond to a Jacobi field constrained to lie in the symplectic orthogonal complement of the space tangent to the geodesics. That is, if there are conjugate points, they arise from coisotropic variations of geodesics. While this is a weak conclusion to obtain from a seemingly strong hypothesis, it is not clear how to show more solely utilizing unsophisticated arguments mimicking those used in studying Riemannian manifolds.

Section \ref{indefinitesection} records a criterion for a symplectic connection to have indefinite signature symplectic sectional curvature and examples are given for certain left-invariant symplectic connections on symplectic Lie groups. There is defined a canonical left-invariant symplectic connection on a symplectic Lie group. It corresponds to one third the symplectically self-adjoint part of the adjoint transformation (see \eqref{saddefined}). Somewhat surprisingly this construction appears to be new, and it will be studied further elsewhere. Here it is used to give an example of a symplectic Lie group for which the canonical symplectic connection has some symplectic sectional curvatures that are degenerate but nonvanishing.

\section{Preliminaries}
Tensors are indicated using the abstract index notation (see, for example, \cite{Penrose-Rindler} or \cite{Wald}). Enclosure of indices in square brackets or parentheses indicates complete skew-symmetrization or complete symmetrization over tensor factors corresponding to the enclosed indices. Indices are in either \textit{up} position or \textit{down}. For sections of tensor powers of a tangent bundle, up indices are interpreted as contravariant and down indices as covariant tensors. For example, $a_{ij}\,^{k}= a_{(ij)}\,^{k} + a_{[ij]}\,^{k}$ indicates the decomposition of $a_{ij}\,^{k} \in \Ga(\tensor^{2}(\ctm) \tensor TM)$ into its symmetric and skew-symmetric parts. The summation convention is always used in the following form: A label appearing as both an up index and a down index indicates the trace pairing, that is, complete contraction with the tautological tensor $\delta_{i}\,^{j}$ determined by the pairing of vectors with covectors. When an index is raised or lowered using some auxiliary tensor, its horizontal position is maintained. 

The antisymmetric bivector $\Om^{ij}$ inverse to $\Om_{ij}$ is defined by $\Om^{ip}\Om_{pj} = -\delta_{j}\,^{i}$. Indices are raised and lowered using $\Omega_{ij}$ and $\Omega^{ij}$ by contracting with these tensors consistently with the conventions $X_{i} = X^{p}\Om_{pi}$ and $X^{i} = \Om^{ip}X_{p}$. Note that $X_{p}Y^{p} = -X^{p}Y_{p}$, so that care must be taken with signs.

The curvature $R_{ijk}\,^{l}$ of a torsion-free affine connection $\nabla$ is defined by $2\nabla_{[i}\nabla_{j]}X^{k} = R_{ijp}\,^{k}X^{p}$ for a vector field $X \in \Ga(TM)$. The \textit{Ricci curvature} of $\nabla$ is defined to be $R_{ij} = R_{pij}\,^{p}$. 

If $\bnabla$ is any torsion-free affine connection on the symplectic manifold $(M, \Om)$ then 
\begin{align}\label{gensym}
\nabla = \bnabla + \tfrac{2}{3}\Om^{kp}\bnabla_{(i}\Om_{j)p}
\end{align}
is symplectic, so every symplectic manifold admits a symplectic connection. The affine space $\symcon(M, \Om)$ of symplectic connections on $(M, \Om)$ is modeled on the vector space $\Ga(S^{3}(\ctm))$ of completely symmetric covariant three tensors, for the difference tensor $\Pi_{ij}\,^{k} = \bnabla - \nabla$ of torsion-free connections satisfies $\Pi_{[ij]}\,^{k} = 0$, and, if $\bnabla, \nabla \in \symcon(M, \Om)$, together with $-2\Pi_{i[jk]} = \bnabla_{i}\Om_{jk} - \nabla_{i}\Om_{jk} = 0$, this implies $\Pi_{ijk} = \Pi_{(ijk)}$.

\section{Curvatures associated with a symplectic connection}
The basic identities for the curvature of a symplectic connection are by now well known and so no effort is made to attribute the derivation of specific identities. The reader is referred to the references mentioned in the introduction and the survey \cite{Bieliavsky-Cahen-Gutt-Rawnsley-Schwachhofer} for background.

For $\nabla \in \symcon(M, \Om)$, since $\Om^{n}$ is $\nabla$-parallel, the curvature $-R_{ijp}\,^{p}$ of the covariant derivative induced by $\nabla$ on $\ext^{2n} \ctm$ vanishes. Consequently, $\nabla$ has symmetric Ricci tensor, for, by the traced algebraic Bianchi identity, $2R_{[ij]} = - R_{ijp}\,^{p} = 0$. Define $R_{ijkl} = R_{ijk}\,^{p}\Om_{pl}$. By the Ricci identity,
\begin{align}\label{curvsym}
0 = 2\nabla_{[i}\nabla_{j]}\Om_{kl} = -2R_{ij[kl]}.
\end{align}
From \eqref{curvsym} there follows
\begin{align}\label{rsym}
4R_{ijkl} = 3R_{i(jkl)} - 3R_{j(ikl)},
\end{align}
so that $R_{i(jkl)}$ completely determines the curvature of $\nabla$. In particular, $R_{ijkl} = 0$ if and only if $R_{i(jkl)} = 0$. (In the context of the Fedosov deformation quantization, the curvature appears naturally in the form $R_{i(jkl)}$; see \cite{Fedosov}.)
The algebraic Bianchi identity and \eqref{curvsym} yield
\begin{align}\label{twotraces}
R_{p}\,^{p}\,_{ij} = -2R_{ip}\,^{p}\,_{j} = 2R_{pij}\,^{p} = 2R_{ij}.
\end{align}
From \eqref{twotraces} it follows that every nontrivial trace of $R_{ijkl}$ is a constant multiple of $R_{ij}$. The differential Bianchi identity and \eqref{twotraces} yield the contracted differential Bianchi identity
\begin{align}\label{symdiffbianchi}
\nabla^{p}R_{pijk} 
= \tfrac{1}{2}\nabla_{i}R_{p}\,^{p}\,_{jk} = \nabla_{i}R_{jk}. 
\end{align}
Although there is no reasonable notion of scalar curvature for $\nabla \in \symcon(M, \Om)$, the \textit{curvature one-form} $\rf_{i}$, defined by $\rf_{i} = \nabla^{p}R_{ip}$, is a direct generalization of the Hamiltonian vector field generated by the scalar curvature of a Kähler metric, that provides, for a general symplectic connection, a useful substitute for the scalar curvature. Precisely, for the Levi-Civita connection of a Kähler metric the vector field metrically dual to $\rf_{i}$ is the negative of the Hamiltonian vector field generated by the scalar curvature. In particular, a Kähler metric has constant scalar curvature if and only if $\rf_{i} = 0$. 

Skew-symmetrizing \eqref{symdiffbianchi} in the indices $ij$ shows that $\nabla_{p}R_{ijk}\,^{p} = -2\nabla_{[i}R_{j]k}$. Hence, if $\nabla_{p}R_{ijk}\,^{p} = 0$ then $\nabla_{[i}R_{j]k} = 0$, and tracing this shows that in this case $\rf_{i} = 0$. 

The \textit{symplectic Weyl tensor} $W_{ijkl}$ of $\nabla \in \symcon(M, \Om)$ is defined to be the part 
\begin{align}\label{symplecticweyl}
W_{ijkl} = R_{ijkl} - \tfrac{1}{n+1}\left(\Om_{i(k}R_{l)j} - \Om_{j(k}R_{l)i} + \Om_{ij}R_{kl} \right).
\end{align}
of $R_{ijkl}$ that vanishes when contracted with $\Om^{ij}$ on any pair of indices. In \cite{Bieliavsky-Cahen-Gutt-Rawnsley-Schwachhofer} a connection with $W_{ijkl} = 0$ is said to be \textit{of Ricci type}, because its curvature is completely determined by its Ricci tensor. Here the alternative terminology \textit{Weyl flat} is used to emphasize the formal parallel with the conformal Weyl tensor of a metric connection. Note that 
\begin{align}\label{wsym}
W_{i(jkl)} = R_{i(jkl)} - \tfrac{2}{n+1}\Om_{i(j}R_{kl)}. 
\end{align}
The identity \eqref{rsym} is valid for any tensor with the symmetries of the curvature tensor of a symplectic connection. In particular, \eqref{rsym} is valid with $W_{ijkl}$ in place of $R_{ijkl}$. It follows that $\nabla$ is Weyl flat if and only if $W_{i(jkl)} = 0$.

From \eqref{curvsym}-\eqref{symdiffbianchi} there follows
\begin{align}\label{divw1}
\begin{split}
2(n+1)\nabla^{p}W_{pijk} 
 & = (2n+1)\nabla_{i}R_{jk} - 3\nabla_{(i}R_{jk)} + \Om_{i(j}\rf_{k)}.
\end{split}
\end{align}
The tensor $ -2\nabla^{p}W_{i(jk)p} = 3(\nabla^{p}W_{pijk} - \nabla^{p}W_{p(ijk)})$ is symmetric in $jk$ and vanishes under complete symmetrization. From \eqref{divw1} there follow 
\begin{align}
  \label{divw}
  \begin{split}
    (n+1)\nabla^{p}W_{p(ijk)} &= (n-1)\nabla_{(i}R_{jk)},\\
2(n+1)\nabla^{p}W_{ijkp} &= -2(2n+1)\nabla_{[i}R_{j]k} - \Om_{ij}\rf_{k} + \Om_{k[i}\rf_{j]}.
  \end{split}
\end{align}
In \cite{Bourgeois-Cahen}, $\nabla \in \symcon(M, \Om)$ is called \textit{preferred} if $\nabla_{(i}R_{jk)} = 0$  (see also \cite{Bieliavsky-Cahen-Gutt-Rawnsley-Schwachhofer}). By \eqref{divw}, a symplectic connection on a symplectic manifold of dimension at least $4$ is preferred if and only if $\nabla^{p}W_{p(ijk)} = 0$. 
Call $\nabla \in \symcon(M, \Om)$ \textit{symplectically flat} if it is Weyl flat and preferred. 
Since here $2n > 2$, by \eqref{divw}, Weyl flat implies symplectically flat, while if $2n = 2$, the Weyl flat condition is vacuous. The definition is made by analogy with the definitions of projectively flat or conformally flat. In every case the contracted differential Bianchi identity yields an identity like \eqref{divw}, which is vacuous in the lowest dimensional case, and in this lowest dimensional regime the definition has to be the vanishing of the part of the covariant derivative of the Ricci tensor appearing in \eqref{divw}. 

Contracting \eqref{symplecticweyl} with the Ricci tensor yields
\begin{align}
\label{riccurv}
\begin{split}
R^{pq}W_{pijq} &= R^{pq}R_{pijq} + \tfrac{1}{n+1}R_{ip}R_{j}\,^{p} + \tfrac{1}{2(n+1)}R^{pq}R_{pq}\Om_{ij}.
\end{split}
\end{align}
Simplifying $\nabla_{i}\rf_{j}$ using \eqref{riccurv} yields
\begin{align}
\label{nablarho}
\begin{split}
\nabla_{i}\rf_{j}& = 2\nabla^{p}\nabla_{i}R_{jp} + 2R^{pq}R_{pijq} + 2R_{ip}R_{j}\,^{p},\\
 & = 2\nabla^{p}\nabla_{i}R_{jp} + 2R^{pq}W_{pijq} + \tfrac{2n}{n+1}\left(R_{ip}R_{j}\,^{p} - \tfrac{1}{2n}R^{pq}R_{pq}\Om_{ij}\right).
\end{split}
\end{align}

Let $(g, J, \Om)$ be a pseudo-Kähler or para-Kähler structure. This means that $\Om_{ij}$ is a symplectic form; $J_{i}\,^{j}$ is a field of integrable endomorphisms of the tangent bundle satisfying $J_{p}\,^{j}J_{i}\,^{p} = \ep \delta_{i}\,^{j}$, where $\ep = -1$ in the pseudo-Kähler case, and $\ep = 1$ in the para-Kähler case; $g_{ij}$ is a pseudo-Riemannian metric, having necessarily split signature in the para-Kähler case; and $J_{i}\,^{p}g_{pj} = \Om_{ij}$. (The qualifier \textit{pseudo} means that the metric need not be Riemannian, although it could be.) Such a structure $(g, J, \Om)$ has \textit{constant (para-)holomorphic sectional curvature $4c$} if the curvature $R_{ijkl} = R_{ijk}\,^{p}\Om_{pl}$ of the Levi-Civita connection $D$ of $g_{ij}$, which is symplectic, satisfies
\begin{align}\label{cholo}
R_{ijkl} = 2c\left( \Om_{l[j}g_{i]k} - g_{l[i}\Om_{j]k} + \Om_{ij}g_{kl}\right) = 2c\left( \Om_{i(k}g_{l)j} - \Om_{j(k}g_{l)i} + \Om_{ij}g_{kl}\right).
\end{align}
(The second equality in \eqref{cholo} is always true). When comparing \eqref{cholo} with formulas found in other sources, it should be kept in mind that indices are raised and lowered using the symplectic form $\Om_{ij}$ and not the metric $g_{ij}$. 

\begin{lemma}\label{chololemma}
Let $D$ be the Levi-Civita connection of a pseudo-Kähler or para-Kähler structure $(g, J, \Om)$ on a symplectic manifold of dimension $2n > 2$. Then:
\begin{enumerate}
\item $D$ has metrically trace-free symplectic Weyl tensor if and only if $g$ is Einstein.
\item $D$ has vanishing symplectic Weyl tensor if and only if it has constant (para-)holomorphic sectional curvature.
\end{enumerate} 
\end{lemma}

\begin{proof}
By \eqref{cholo},
\begin{align}\label{metricw}
R_{ijk}\,^{p}g_{pl} = W_{ijk}\,^{p}g_{pl} + \tfrac{1}{n+1}\left(g_{l[i}R_{j]k} + \ep \Om_{k[i}R_{j]p}J_{l}\,^{p} - \ep\Om_{ij}R_{kp}J_{l}\,^{p}\right).
\end{align}
Contracting this with $g^{jk}$ and simplifying the result yields
\begin{align}\label{traceweinstein}
g^{ab}W_{iab}\,^{p}g_{pj} = \tfrac{n}{n+1}(R_{ij} - \tfrac{1}{2n}R_{g}g_{ij}),
\end{align}
where $R_{g} = g^{ij}R_{ij}$ is the (metric) scalar curvature of $g$. From \eqref{traceweinstein} and $W_{[ijk]l} = 0$ it follows that $W_{ijk}\,^{l}$ is metrically trace-free if and only if $g$ is Einstein. If $D$ has constant (para-)holomorphic sectional curvature, then tracing \eqref{cholo} shows that $R_{ij} = 2(n+1)cg_{ij}$, and substituting this into \eqref{cholo} and comparing with \eqref{symplecticweyl} shows that $W_{ijkl} = 0$. If $D$ satisfies $W_{ijkl} = 0$, then by \eqref{traceweinstein}, $g$ is Einstein with $2nR_{ij} = R_{g}g_{ij}$. Since the scalar curvature of an Einstein pseudo-Riemannian metric is constant, substituting this into \eqref{symplecticweyl} yields \eqref{cholo} with $R_{g} = 4n(n+1)c$.
\end{proof}

An endomorphism $A_{i}\,^{j}$ of the tangent bundle of $(M, \Om)$ is infinitesimally symplectic, meaning $2A_{[ij]} = A_{i}\,^{p}\Om_{pj} + A_{j}\,^{p}\Om_{ip} = 0$, if and only if the associated tensor $A_{ij} = A_{i}\,^{p}\Om_{pj}$ is symmetric. 

\begin{lemma}\label{symflatlemma}
Let $(M, \Om)$ be a $2n$-dimensional symplectic manifold. For a symplectically flat connection $\nabla \in \symcon(M, \Om)$, the following are equivalent:
\begin{enumerate}
\item\label{sss1} The curvature one-form vanishes.
\item\label{sss2} The Ricci tensor is parallel. 
\item\label{sss3} $\nabla$ is locally symmetric.
\end{enumerate}
In the case there hold \eqref{sss1}-\eqref{sss3}, the Ricci endomorphism $R_{i}\,^{j}$ satisfies
\begin{align}\label{getcomplex}
R_{j}\,^{p}R_{p}\,^{i} = -r\delta_{j}\,^{i},
\end{align}
where $2nr = R_{pq}R^{pq}$ is constant. Moreover:
\begin{enumerate}\renewcommand{\theenumi}{\roman{enumi}}
\item If $r > 0$, then $J_{i}\,^{j} = |r|^{-1/2}R_{i}\,^{j}$ is a complex structure forming with $\Om_{ij}$ a constant holomorphic sectional curvature pseudo-Kähler structure with associated pseudo-Riemannian metric $-|r|^{-1/2}R_{ij}$ and Levi-Civita connection $\nabla$.
\item If $r < 0$, then $J_{i}\,^{j} = |r|^{-1/2}R_{i}\,^{j}$ is a paracomplex structure forming with $\Om_{ij}$ a  constant paraholomorphic sectional curvature para-Kähler structure with associated pseudo-Riemannian metric $|r|^{-1/2}R_{ij}$ and Levi-Civita connection $\nabla$.
\item If $r = 0$ there holds one of the following:
\begin{enumerate}
\item The Ricci curvature $R_{ij}$ is identically zero and $\nabla$ is a flat affine connection.
\item The Ricci endomorphism $R_{i}\,^{j}$ is preserved by $\nabla$, is infinitesimally symplectic, and is two-step nilpotent. Its kernel equals the symplectic annihilator of its image, and its image and kernel are totally geodesic integrable subbundles of $TM$. Moreover the induced connections on these integrable subbundles are flat.
\end{enumerate}
\end{enumerate}
\end{lemma}

The equivalence of \eqref{sss1} and \eqref{sss3} in Lemma \ref{symflatlemma} is the corollary to Lemma $1$ in \cite{Cahen-Gutt-Horowitz-Rawnsley}. As is explained in Remark \ref{ssremark} following the proof, Lemma \ref{symflatlemma} can also be deduced from Theorem $2$ of \cite{Cahen-Gutt-Rawnsley-symmetric}.
\begin{proof}
By \eqref{divw1} and \eqref{divw}, that $W_{ijkl} = 0$ implies $(2n+1)\nabla_{i}R_{jk} = - \Om_{i(j}\rf_{k)}$ (this is Lemma $1$ of \cite{Cahen-Gutt-Horowitz-Rawnsley}). It follows that $R_{ij}$ is parallel if $\rf_{i} = 0$ (the converse is true by definition of $\rf_{i}$). From the contracted differential Bianchi identity \eqref{symdiffbianchi} it follows that if $\nabla$ is locally symmetric, then the Ricci tensor is parallel. On the other hand, differentiating \eqref{symplecticweyl} shows that if $W_{ijkl} = 0$ and the Ricci tensor is parallel, then $\nabla$ is locally symmetric. In the case there holds these equivalent conditions, equation \eqref{nablarho} implies \eqref{getcomplex}. Since $R_{ij}$ is parallel, $2n r = R_{pq}R^{pq}$ is constant. If $r \neq 0$, then \eqref{getcomplex} implies $R_{i}\,^{j}$ is invertible, so that $R_{ij}$ is nondegenerate. The claims when $r \neq 0$ follow from \eqref{getcomplex}, the observation that a parallel endomorphism has vanishing Nijenhuis tensor, and the conclusion of Lemma \ref{chololemma} that the Levi-Civita connection $D$ of a pseudo-Kähler or para-Kähler structure $(g, J, \Om)$ on a symplectic manifold of dimension $2n > 2$ has vanishing symplectic Weyl tensor if and only if it has constant (para-)holomorphic sectional curvature. 

Suppose $r = 0$. Temporarily write $\rice_{i}\,^{j} = R_{i}\,^{j}$ for the Ricci endomorphism. That $\ker \rice = (\im \rice)^{\perp}$ follows from the fact that $\rice$ is infinitesimally symplectic (equivalent to the symmetry of $R_{ij}$) and the relation $\rice \circ \rice = 0$ (this conclusion was observed in section $4.3$ of \cite{Cahen-Gutt-Horowitz-Rawnsley}). Since $\nabla \rice = 0$, $\nabla_{X}(\rice Y) = \rice(\nabla_{X}Y)$ for all $X, Y \in \Ga(TM)$, from which it follows that $\im \rice$ and $\ker \rice$ are totally geodesic. Because $[X, Y] = \nabla_{X}Y - \nabla_{Y}X$, these subbundles are also integrable. It follows from \eqref{symplecticweyl} that the connections induced on these subbundles are flat.
\end{proof}

\begin{remark}\label{ssremark}
The conclusions of Lemma \ref{symflatlemma} are very similar to those of Theorem $2$ of \cite{Cahen-Gutt-Rawnsley-symmetric}, where they are deduced for symplectic symmetric spaces as defined by P. Bieliavsky in his thesis \cite{Bieliavsky-thesis} (see also \cite{Bieliavsky-Cahen-Gutt} or \cite{Bochenski-Tralle}). A \textit{symmetric symplectic space} is a symplectic manifold $(M, \Om)$ with a symmetric space structure such that for every $p \in M$ the symmetry involution $s_{p}$ is a symplectomorphism. In \cite{Bieliavsky-Cahen-Gutt} it is proved that on a symmetric symplectic space there is a canonical symplectic connection $\nabla$ that makes $M$ an affine symmetric space and such that the automorphisms of the symmetric symplectic space are exactly the automorphisms of $\nabla$ that are also symplectomorphisms. To deduce Lemma \ref{symflatlemma} from Theorem $2$ of \cite{Cahen-Gutt-Rawnsley-symmetric} it suffices to show that an affine symmetric space that admits a symplectic form preserved by the affine connection is necessarily a symmetric symplectic space. This amounts to showing that in this case the symmetry involutions are necessarily symplectomorphisms. This can be shown in the same way as the analogous fact for Riemannian symmetric spaces, as in Section $6$ of Chapter XI of \cite{Kobayashi-Nomizu-2}:
\begin{lemma}
Let $(M, \Om)$ be a symplectic manifold. Suppose that a symplectic connection $\nabla \in \symcon(M, \Om)$ is affine locally symmetric. Then for each $p \in M$ the symmetry $s_{p}$ is a symplectomorphism. 
\end{lemma}
\begin{proof}
First, it will be shown that if a diffeomorphism $\phi$ of a symplectic manifold $(M, \Om)$ is an automorphism of a symplectic connection $\nabla \in \symcon(M, \Om)$ and there is a point $p \in M$ such that $\phi^{\ast}(\Om)_{p} = \Om_{p}$, then $\phi$ is a symplectomorphism. Let $q \in M$ and let $\si$ be a curve in $M$ from $q$ to $p$. Let $P$ and $P^{\prime}$ be the operators of parallel transport along $\si$ and $\si^{\prime} = f(\si)$. For $X, Y \in T_{q}M$,
\begin{align}
\begin{split}
\Om_{f(q)}&(TF(q)(X), TF(q)(Y))  = \Om_{f(p)}(P^{\prime} Tf(q)(X), P^{\prime} Tf(q)(Y)) \\
&= \Om_{f(p)}(Tf(p)(PX), Tf(p)(PY)) = f^{\ast}(\Om)_{p}(PX, PY) = \Om_{p}(PX, PY) = \Om_{q}(X, Y),
\end{split}
\end{align}
showing that $f^{\ast}(\Om) = \Om$. As a consequence of this claim, if $\nabla$ is affine locally symmetric and preserves $\Om$, then every symmetry involution $s_{p}$ preserves $\Om$ (because $Ts_{p}(p) = -\Id$, so preserves $\Om$), and so $\nabla$ and $\Om$ determine a structure of a locally symmetric symplectic space.
\end{proof}
\end{remark}

\section{Symplectic sectional curvature}\label{curvaturesection}
Let $\gr(2, TM)$ be the bundle over $M$ the fibers of which are the Grassmannians of two-dimensional subspaces of $TM$. A field of two-dimensional subspaces $L \subset \gr(2, TM)$ is \textit{symplectic} if the restriction of $\Om_{ij}$ to $L$ is nondegenerate. Let $\sgr(2, TM) \subset \gr(2, TM)$ be the open subset the fibers of which comprise symplectic subspaces. 

Let $\nabla \in \symcon(M, \Om)$. Because $R_{ijkl}$ is symmetric in $k$ and $l$ it makes sense to consider the map $\Ga(\ext^{2}(TM)) \to \Ga(S^{2}(\ctm))$ defined by $\om^{ij} \to \om^{ij}R_{ijkl}$. Identify $\om^{ij}R_{ijkl}$ with the quadratic form $\Q_{\om}(Z) = \om^{ij}R_{ijkl}Z^{k}Z^{l}$. A bivector $\om^{ij} \in \Ga(\ext^{2}(TM))$ is \textit{symplectic} if $\om^{ij}\Om_{ij} \neq 0$. An analogue of the sectional curvature of a Riemannian metric is obtained by considering $\Q_{\om}$ for decomposable symplectic bivectors $X \wedge Y$. Specifically, given $p \in M$ and a symplectic subspace $L \in \sgr(2, T_{p}M)$, the quadratic form $\scurv_{p, L}$ defined by 
\begin{align}\label{scurvdefined}
\scurv_{p, L}(Z) = \frac{ X^{i}Y^{j}Z^{k}Z^{l}R_{ijkl}}{X^{i}Y^{j}\Om_{ij}},
\end{align}
for any decomposable symplectic bivector $2X^{[i}Y^{j]}$ spanning $L$ and any $Z \in T_{p}M$, does not depend on the choice of $X$ and $Y$. It is the \textit{symplectic curvature quadratic form} of $L$. The \textit{symplectic sectional curvature} of $L \in \sgr(2, T_{p}M)$ means the quadratic form on $L$ obtained by restricting $\scurv_{p, L}$ to $L$. This notion was proposed by I.~M. Gelfand, V. Retakh, and M. Shubin in \cite{Gelfand-Retakh-Shubin}. 
The possibilities at $p$ are described by the $SL(2, \rea)$ orbits on the space of quadratic forms on $\rea^{2}$. The nontrivial possibilities are: definite, with a repeated eigenvalue; indefinite with additively inverse eigenvalues; and degenerate with single eigenvalue $\pm 1$. 

The symplectic connection $\nabla$ has positive (resp. negative, nonnegative, etc.) \textit{symplectic sectional curvature} if for every $p \in M$ and every $L \in \sgr(T_{p}M)$ the restriction to $L$ of the quadratic form $\scurv_{p, L}$ is positive definite (resp. negative definite, nonnegative definite, etc.). 

For a pseudo-Kähler or para-Kähler metric $(g, J, \Om)$ with $J_{i}\,^{p}J_{p}\,^{j} = \ep \delta_{i}\,^{j}$, the (para-)holomorphic sectional curvature at $p \in M$ of the span $L \subset T_{p}M$ of a nonnull vector $X \in T_{p}M$ and its image $JX$ is defined to be 
\begin{align}\label{hcurvdefined}
\hcurv_{p, L} =   \tfrac{g(R(JX, X)X, JX)}{g(X, X)g(JX, JX) - g(X, JX)^{2}} = -\ep \tfrac{g(R(JX, X)X, JX)}{g(X, X)^{2}} =  -\ep \tfrac{\Om(R(JX, X)X, X)}{g(X, X)^{2}} .
\end{align}
The definition \eqref{hcurvdefined} agrees with the previously made definition of constant (para-)holomorphic sectional curvature; if the curvature has the form \eqref{cholo} then the expression \eqref{hcurvdefined} evaluates to $4c$.

Lemma \ref{grslemma} extends the observation of Remark $3.13$ of \cite{Gelfand-Retakh-Shubin}, that a symplectic sectional curvature of a Kähler metric is either definite or identically zero.
\begin{lemma}\label{grslemma}
For a pseudo-Kähler or para-Kähler structure $(g, J, \Om)$ with $J_{i}\,^{p}J_{p}\,^{j} = \ep \delta_{i}\,^{j}$, the symplectic sectional curvature of the span $L$ of a nonnull vector $X$ and its image $JX$ equals the (para-)holomorphic sectional curvature of $L$ multiplied by the restriction to $L$ of $g$. Precisely, 
\begin{align}\label{kahlersec2}
\scurv_{p, L}(Y) = g(Y, Y)\hcurv_{p, L}
\end{align}
for $Y$ contained in $L$.
\end{lemma}
\begin{proof}
There holds $\Om(R(X, JX)Y, Y) = - g(R(X, JX)Y, JY)$.
The symplectic sectional curvature of the symplectic subspace spanned by $X$ and $JX$ is simply the product of the (para-)holomorphic sectional curvature of the span of $X$ and $JX$ with the quadratic form on $\spn\{X, JX\}$ determined by $g$. To make this precise, define
\begin{align}
&s(\theta) = \begin{cases} \sin \theta & \text{if}\,\, \ep = -1,\\ \sinh \theta & \text{if}\,\, \ep = 1\end{cases},&& c(\theta) = 
 \begin{cases} \cos \theta & \text{if}\,\, \ep = -1,\\ \cosh \theta & \text{if}\,\, \ep = 1\end{cases}
\end{align}
If $Y = r(c(\theta) X + s(\theta) JX)$ then $g(Y, Y) = r^{2}(c(\theta)^{2} - \ep s(\theta)^{2})g(X, X) = r^{2}g(X, X)$, and 
\begin{align}\label{kahlersec}
\Om(R(X, JX)Y, Y) = r^{2}g(R(JX, X)X, JX),
\end{align}
where there has been used $R(X, JX)JX = JR(X, JX)X$.
Dividing \eqref{kahlersec} by $\Om(X, JX)$ and comparing the result with \eqref{hcurvdefined} yields \eqref{kahlersec2}.
\end{proof}
One consequence of Lemma \ref{grslemma} is the conclusion of Remark $3.13$ of \cite{Gelfand-Retakh-Shubin}, that the symplectic sectional curvatures of the Levi-Civita connection of a Kähler metric either have definite signature or are identically zero.

The identity \eqref{kahlersec} shows that the symplectic sectional curvatures carry at least as much information as do the holomorphic sectional curvatures. In particular, if the symplectic sectional curvatures of a Kähler form are positive or negative, the same is true of its holomorphic sectional curvatures. 

Theorem \ref{scurvtheorem} shows the relation between the notion of symplectic sectional curvature and the symplectic Weyl tensor.

\begin{theorem}\label{scurvtheorem}
On a $2n$-dimensional symplectic manifold $(M, \Om)$ a symplectic connection $\nabla \in \symcon(M, \Om)$ is Weyl flat if and only if there is a symmetric covariant two tensor $A_{ij} = A_{(ij)}$ such that for every $p \in M$ and every two-dimensional symplectic subspace $L \subset T_{p}M$ there holds $\scurv_{p, L}(Z) = 4A(Z, Z)$ whenever $Z \in L$. Moreover, if this is the case, then in fact $2(n+1)A_{ij} = R_{ij}$.
\end{theorem}
\begin{proof}
The projection $P_{L}:\Ga(TM) \to \Ga(TM)$ onto the two-dimensional symplectic subspace $L$ along its symplectic complement $L^{\perp}$ is given in terms of any vector fields $X$ and $Y$ spanning $L$ by 
\begin{align}
P_{L}(Z) = \Om(X, Y)^{-1}\left(\Om(Z, Y)X - \Om(Z, X)Y\right).
\end{align}
If $\nabla$ is Weyl flat then, writing $R_{ij} = 2(n+1)A_{ij}$, 
for any $L \in \sgr(2, TM)$ and $Z \in \Ga(TM)$, there holds $\scurv_{L}(Z) = 2A(Z + P_{L}(Z), Z)$, by \eqref{scurvdefined}. It follows that if $Z$ is in the symplectic complement $L^{\perp}$ then $\scurv_{L}(Z) = 2A(Z, Z)$, while if $Z$ is in $L$ then $\scurv_{X, Y}(Z) = 4A(Z, Z)$. 

Now suppose that for any $L \in \sgr(2, TM)$ there holds $\scurv_{L}(Z) = 4A(Z, Z)$ whenever $Z \in L$. Then, by the definition \eqref{scurvdefined}, there holds
\begin{align}\label{scd1}
\Om(R(Y, X)X, X) = -4A(X, X)\Om(X, Y)
\end{align}
whenever $\Om(X, Y) \neq 0$. Fix $X$ and choose $U \in \Ga(TM)$ such that $\Om(X, U) = 1$. Then any vector $Y$ has the form $Y = fU + B$ for some $B$ such that $\Om(X, B) = 0$ and some function $f$. Applying \eqref{scd1} as is, and with $U$ in place of $Y$, yields $\Om(R(B, X)X, X) = 0$. Hence \eqref{scd1} holds for all $Y$. Since $X$ is arbitrary this shows $R_{i(jkl)} =- 4A_{(jk}\Om_{l)i}$.
By \eqref{curvsym} there holds $4R_{ijkl} = 3R_{i(jkl)} - 3R_{j(ikl)}$ for any symplectic connection. Hence, 
\begin{align}\label{ra}
\begin{split}
R_{ijkl} & = \tfrac{3}{4}\left(R_{i(jkl)} - 3R_{j(ikl)}\right) = 2\left(\Om_{i(k}A_{l)j} - \Om_{j(k}A_{l)i} + \Om_{ij}A_{kl}\right).
\end{split}
\end{align}
Tracing \eqref{ra} shows $R_{ij} = 2(n+1)A_{ij}$, and, substituted into \eqref{ra}, this shows $\nabla$ is Weyl flat.
\end{proof}

Because of Theorem \ref{scurvtheorem}, it makes sense to say that $\nabla \in \symcon(M, \Om)$ has \textit{constant symplectic sectional curvature} if there is a parallel symmetric covariant $2$-tensor $A_{ij} = A_{(ij)}$ such that for all $p$ and all $L \in \sgr(2, T_{p}M)$, the quadratic form $\scurv_{p, L}$ is equal to the restriction to $L$ of the quadratic form determined by $A_{ij}$. 
\begin{corollary}\label{constantcurvcorollary}
A symplectic connection $\nabla$ has constant symplectic sectional curvature if and only if it is Weyl flat with parallel Ricci tensor, in which case it is locally symmetric.
\end{corollary}
\begin{proof}
By Theorem \ref{scurvtheorem}, a symplectic connection with constant symplectic sectional curvatures is Weyl flat, and its Ricci tensor is parallel, so is locally symmetric by Lemma \ref{symflatlemma}. 
\end{proof}

Formula \eqref{ra} shows that a Kähler manifold of constant holomorphic sectional curvature $4c$ has constant symplectic sectional curvature $4cg$. In general, the differences with the Kähler case are that $A_{ij}$ need not be nondegenerate, and $A_{i}\,^{j}$ need not be a complex structure. However, together Lemma \ref{symflatlemma} and Corollary \ref{constantcurvcorollary} imply the following.
\begin{corollary}\label{ncpccorollary}
On the $2n$-dimensional symplectic manifold $(M,\Om)$, the connection $\nabla \in \symcon(M, \Om)$ has constant symplectic sectional curvature if and only if one of the following mutually exclusive possibilities occurs, where $2nr = R_{pq}R^{pq}$:
\begin{enumerate}
\item $r = 0$ and $R_{i}\,^{j}$ is a parallel integrable nilpotent endomorphism of square zero.
\item $r > 0$ and either of $\pm |r|^{-1/2}R_{i}\,^{j}$ is a complex structure forming with $\Om$ a constant homolomorphic sectional curvature pseudo-Kähler structure with Levi-Civita connection $\nabla$.
\item $r < 0$ and either of $\pm |r|^{-1/2}R_{i}\,^{j}$ is a paracomplex structure forming with $\Om_{ij}$ a constant paraholomorphic sectional curvature para-Kähler structure with Levi-Civita connection $\nabla$.
\end{enumerate}
\end{corollary}

\begin{corollary}\label{chcorollary}
Let $(M, \Omega)$ be a $2n$-dimensional symplectic manifold and let $\nabla \in \symcon(M, \Om)$. 
\begin{enumerate}
\item If either $M$ is compact or $\nabla$ is complete, then $\nabla$ has constant positive symplectic sectional curvature if and only if $M$ is complex projective space and $\nabla$ is the Levi-Civita connection of the Fubini-Study metric.
\item $\nabla$ has constant negative symplectic sectional curvature if and only if $\nabla$ is the Levi-Civita connection of a complex hyperbolic metric.
\end{enumerate}
\end{corollary}
\begin{proof}
If $\nabla$ has constant symplectic sectional curvature which is definite, then either $R_{ij}$ or $-R_{ij}$ is a Riemannian metric with Levi-Civita connection $\nabla$. Since $2nr = R^{pq}R_{pq} = R_{ab}R_{pq}\Om^{ap}\Om^{aq}$ is simply the norm of the bivector $\Om^{ij}$ in this metric, in either case $r$ is a positive constant. By Lemma \ref{symflatlemma}, one of $\pm |r|^{-1/2}R_{i}\,^{j}$ is a complex structure forming with $\Om_{ij}$ a Kähler structure which by Lemma \ref{grslemma} has constant holomorphic sectional curvature of the same sign as the symplectic sectional curvature of $\nabla$. The Levi-Civita connections of homothetic metrics are the same, so $\nabla$ is the Levi-Civita connection of a Kähler metric of constant holomorphic sectional curvature $\pm 4$. In the positive case, if $\nabla$ is assumed complete, then, by the Myers Theorem, it must in fact be compact, and then, by the Synge Theorem, it is simply-connected. The only simply-connected manifold admitting a metric of constant positive holomorphic sectional curvature is complex projective space. On the other hand, that the Fubini-Study metric has the stated properties is clear. 
\end{proof}

\section{Isotropic sectional curvature}\label{isotropicsection}
Compared with the Riemannian setting, one of the difficulties in working with symplectic sectional curvature is the existence of isotropic subspaces, so it is reasonable to look for analogues with the sectional curvature of indefinite signature metrics, which also admit isotropic subspaces. 

A Lorentzian metric $g$ with curvature tensor $R$ on a manifold $M$ is said to have \textit{vanishing null sectional curvature} at $p \in M$ if $g(R(X, Y)X, Y) = 0$ for all vectors $X, Y \in T_{p}M$ spanning a $g$-null subspace. It is known that a Lorentzian metric on a manifold of dimension at least $3$ has vanishing null sectional curvature at every point if and only if it has constant sectional curvature (where sectional curvatures of a Lorentzian metric are defined only for nondegenerate subspaces). See, for example, Proposition $8.28$ of \cite{O'Neill-book}.

The analogy with the Lorentzian situation motivates the following definition and Lemma \ref{nulllemma}. Let $(M, \Om)$ be a $2n$-dimensional symplectic manifold. A symplectic connection $\nabla \in \symcon(M, \Om)$ has \textit{vanishing isotropic sectional curvature} at $p \in M$ if $\Om(R(X, Y)Y, Y) = 0$ for all vectors $X, Y \in T_{p}M$ spanning an isotropic subspace of $T_{p}M$, that is, such that $\Om (X, Y) = 0$.

\begin{lemma}\label{nulllemma}
 Let $(M, \Om)$ be a symplectic manifold of dimension $2n \geq 4$. A symplectic connection $\nabla \in \symcon(M, \Om)$ has constant symplectic sectional curvature if and only if it has vanishing isotropic sectional curvature at every point of $M$.
\end{lemma}

\begin{proof}
If $\nabla$ has constant symplectic sectional curvature, then by Corollary \ref{constantcurvcorollary} there holds $W_{ijkl} = 0$. Consequently if $\Om_{ij}X^{i}Y^{j} = 0$, by \eqref{wsym} there holds
\begin{align}
R_{ijkl}X^{i}Y^{j}Y^{k}Y^{l} = (W_{ijkl} + \tfrac{2}{n+1}\Om_{ij}R_{kl})X^{i}Y^{j}Y^{k}Y^{l} = 0,
\end{align}
showing that $\nabla$ has vanishing isotropic sectional curvature. 

Suppose that $\nabla$ has vanishing isotropic sectional curvature. Then 
\begin{align}
W_{ijkl}X^{i}Y^{j}Y^{k}Y^{l} = (R_{ijkl} - \tfrac{2}{n+1}\Om_{ij}R_{kl})X^{i}Y^{j}Y^{k}Y^{l} = R_{ijkl}X^{i}Y^{j}Y^{k}Y^{l} = 0
\end{align} 
for all $X$ and $Y$ such that $X^{i}Y^{j}\Om_{ij}$. Fix a symplectic vector space $(\ste, \Om)$ of dimension $2n \geq 4$. The vector space
\begin{align}
\W = \W(\ste, \Om) = \{A_{ijkl} \in \tensor^{4}\ste^{\ast}: A_{ijkl} = A_{i(jkl)}, A_{(ij)kl} = 0, A_{p}\,^{p}\,_{kl} = 0\}
\end{align}
is an irreducible $Sp(\ste, \Om)$-module, where $Sp(\ste, \Om)$ is the group of linear symplectic automorphisms of $(\ste, \Om)$. The subspace
\begin{align}
\Z = \{A_{ijkl} \in \W: A_{ijkl}X^{i}Y^{j}Y^{k}Y^{l} = 0 \,\,\text{for all}\,\, X \,\,\text{and} \,\, Y \,\,\text{such that}\,\, X^{i}Y^{j}\Om_{ij} = 0\}
\end{align}
is $Sp(\ste, \Om)$ invariant, so must equal either $\{0\}$ or $\W$. Since $2n \geq 4$, there exist $X, Y, U, V \in \ste$ such that $\Om(X, Y) = 0$, $\Om(U, X) = 1$, $\Om(U, Y) = 0$, $\Om(V, X) = 0$, $\Om(V, Y) = 0$, and $\Om(U, V) = 0$ (so $\{X, Y, U, V\}$ is a symplectic basis of its span). Then $A_{ijkl} = X_{i}Y_{j}Y_{k}Y_{l} - Y_{i}X_{j}X_{k}X_{l} \in \W$. As $A_{ijkl}U^{i}V^{j}V^{k}V^{l} = 1$, $A_{ijkl} \notin \Z$. Thus $\Z$ is a proper subspace of $\W$, so $\Z = \{0\}$. Since $W_{ijkl}$ is contained in $\Z(T_{p}M, \Om)$ for every $p \in M$, it follows that $W_{ijkl} = 0$. By Corollary \ref{constantcurvcorollary}, $\nabla$ has constant symplectic sectional curvature.
\end{proof}

Via the analogy with Lorentzian metrics, Lemma \ref{nulllemma} provides evidence that the proposed notion of constant symplectic sectional curvature is reasonable.

\section{Symplectic sectional curvature of a symplectic submanifold}\label{submanifoldsection}
The differential $T\phi$ of a smooth map $\phi:N \to M$ between the smooth manifolds $N$ and $M$ can be viewed as a section of the tensor product $T^{\ast}N \tensor \phi^{\ast}TM$ of the cotangent bundle of $N$ and the pullback $\phi^{\ast}(TM) \to N$ of the tangent bundle $TM$ via $\phi$. A torsion-free affine connection $\nabla$ on $M$ induces a connection on the pullback bundle $\phi^{\ast}TM$. If $\bnabla$ is a torsion-free connection on $N$, then the product connection induced on the tensor product of any tensor power of $TN$ and $T^{\ast}N$ with $\phi^{\ast}TM$ is the connection denoted $\blD$ and defined by $\blD(\al \tensor s) = \bnabla\al \tensor s + \al \tensor \nabla s$ where $\al$ is a tensor on $N$ and $s$ is a section of $\phi^{\ast}TM$. The \textit{second fundamental form} $\Pi$ of the smooth map $\phi:N \to M$ with respect to connections $\nabla$ and $\bnabla$ is defined to be the differential $\Pi = \blD T\phi$. By definition, for vector fields $X$ and $Y$ on $N$,
\begin{align}\label{bldf}
\Pi(X, Y) = \blD_{X}T\phi(Y) = \nabla_{X}T\phi(Y) - T\phi(\bnabla_{X}Y),
\end{align}
where the right-hand side is interpreted rigorously in terms of extensions of $T\phi(X)$ and $T\phi(Y)$ defined near $\phi(N)$. Since $\nabla$ and $\bnabla$ are torsion-free, $\blD_{X}T\phi(Y) - \blD_{Y}T\phi(X)  = [T\phi(X), T\phi(Y)] - T\phi([X, Y]) = 0$, so $\Pi$ is symmetric.

Lowercase Latin indices label sections of tensor powers of $TN$ and its dual and uppercase Latin indices label sections of tensor powers of $\phi^{\ast}TM$ and its dual. With these conventions, the differential $T\phi:TM \to TN$, viewed as a section of $T^{\ast}M \tensor \phi^{\ast}TM$, is written $\phi_{i}\,^{A}$, and the second fundamental form of $\phi$ with respect to $\nabla$ and $\bnabla$ is written $\Pi_{ij}\,^{A}$.

Let $(M, \Om)$ be a $2n$-dimensional symplectic manifold. A smooth map $\phi:N \to M$ is \textit{symplectic} if $\phi^{\ast}(\Om)$ is a symplectic form on $N$; in this case $\phi$ is necessarily an immersion and $N$ is necessarily even-dimensional. Let $\phi_{\ast}(TN)$ be the subbundle of $\phi^{\ast}TM$ the fiber over $p \in N$ of which is $T\phi(p)(T_{p}N)$. The bundle $\phi^{\ast}TM$ splits as the fiberwise direct sum $\phi^{\ast}TM = \phi_{\ast}(TN) \oplus \phi_{\ast}(TN)^{\perp}$ where the vector bundle $\phi_{\ast}(TN)^{\perp}$ is the fiberwise orthogonal complement of $\phi_{\ast}(TN)$ in $\phi^{\ast}TM$ with respect to the symplectic form, also denoted $\Om$, induced on $\phi^{\ast}TM$ by $\Om$. 

Consider a symplectic manifold $(M, \Om)$, a smooth manifold $N$, a symplectic connection $\nabla \in \symcon(M, \Om)$, a torsion-free affine connection $\bnabla$ on $N$, and a smooth map $\phi:N \to M$ with second fundamental form $\Pi_{ij}\,^{A}$ with respect to $\nabla$ and $\bnabla$. By \eqref{bldf}, for any $Q_{i_{1}, \dots, i_{k}} \in \Ga(\tensor^{k}\ctm)$,
\begin{align}\label{bnpq}
\begin{split}
\bnabla_{i}&\phi^{\ast}(Q)_{j_{1} \dots j_{k}} -  \phi^{\ast}(\nabla Q)_{ij_{1}\dots j_{k}} \\
& =  \sum_{s = 1}^{k}\Pi_{ij_{s}}\,^{A}\phi_{j_{1}}\,^{B_{1}}\dots \phi_{j_{s-1}}\,^{B_{s-1}}\phi_{j_{s+1}}\,^{B_{s+1}}\dots \phi_{j_{k}}\,^{B_{k}}Q_{B_{1}\dots B_{s-1} A B_{s+1} \dots B_{k}}.
\end{split}
\end{align}
Define a section $\ind$ of $S^{2}(T^{\ast}N) \tensor T^{\ast}N$ by $\ind(X, Y, Z) = \Om(\Pi(X, Y), T\phi(Z))$. That is, $\ind_{ijk} = \Pi_{ij}\,^{A}\phi_{k}\,^{B}\Om_{AB}$. By \eqref{bnpq},
\begin{align}\label{bninduced}
\bnabla_{i}\phi^{\ast}(\Om)_{jk} = \Pi_{ij}\,^{A}\phi_{k}\,^{B}\Om_{AB} - \Pi_{ik}\,^{A}\phi_{j}\,^{B}\Om_{AB} = 2\ind_{i[jk]},
\end{align}
so $\bnabla$ preserves $\phi^{\ast}(\Om)$ if and only if $\ind_{ijk} = \Pi_{ij}\,^{A}\phi_{k}\,^{B}\Om_{AB}$ is completely symmetric. If $\phi$ is moreover a symplectic immersion, the connection $\bnabla$ \textit{induced} on $N$ by $\nabla \in \symcon(M, \Om)$ and $\phi$ is defined by requiring that $T\phi(\bnabla_{X}Y)$ be the symplectic orthogonal projection on $\phi_{\ast}(TN)$ of $\nabla_{X}T\Phi(Y)$. It is easily checked that $\bnabla$ is torsion-free. It follows from the definition of the induced connection $\bnabla$ that the second fundamental form of $\phi$ with respect to $\nabla$ and $\bnabla$ is the symplectic orthogonal projection on $\phi_{\ast}(TN)^{\perp}$ of $\nabla_{X}T\Phi(Y)$ and so satisfies $\ind_{ijk} = \Pi_{ij}\,^{A}\phi_{k}\,^{B}\Om_{AB} = 0$. It follows from \eqref{bninduced} that the induced connection $\bnabla$ preserves $\phi^{\ast}(\Om)$, so is an element of $\symcon(N, \phi^{\ast}(\Om))$.

Let $\phi:N \to (M, \Om)$ be a symplectic immersion. Let $\nabla \in\symcon(M, \Om)$ and let $\bnabla$ be a torsion-free affine connection on $N$. Let $\Pi$ be the second fundamental form of $\phi$ with respect to $\nabla$ and $\bnabla$. For vector fields $X$ and $Y$ on $N$ let $\sR(X, Y)$ be the curvature of the connection $\nabla$ induced on $\phi^{\ast}TM$ and let $\bar{R}(X, Y)$ be the curvature of $\bnabla$. Note that while $\nabla$ is used to indicate both the given symplectic connection on $M$ and the induced connection on $\phi^{\ast}TM$, the curvatures of these connections are distinguished as $R(\dum, \dum)$ and $\sR(\dum, \dum)$, as they are sections of different bundles. Precisely, $\sR(X, Y) = R(T\phi(X), T\phi(Y))$, or, equivalently, $\sR_{ijA}\,^{B} = \phi_{i}\,^{C}\phi_{j}\,^{D}R_{CDA}\,^{B}$.

For vector fields $X$, $Y$, and $Z$ on $N$, straightforward calculation shows
\begin{align}\label{gauss}
\begin{split}
\sR(X, Y)T\phi(Z)  &= T\phi(\bar{R}(X, Y)Z) + \nabla_{X}\Pi(Y, Z) - \nabla_{Y}\Pi(X, Z) \\
&\qquad + \Pi(X, \bnabla_{Y}Z) - \Pi(Y, \bnabla_XZ) - \Pi([X, Y], Z)\\
&=  T\phi(\bar{R}(X, Y)Z) + (\blD_{X}\Pi)(Y, Z) - (\blD_{Y}\Pi)(X, Z),
\end{split}
\end{align}
where $\blD\Pi$ is the covariant derivative of $\Pi$ viewed as a section of $S^{2}(T^{\ast}N) \tensor \phi^{\ast}TM$. Alternatively, and more compactly,
\begin{align}
\sR_{ijB}\,^{A}\phi_{k}\,^{B} = \bar{R}_{ijk}\,^{p}\phi_{p}\,^{A} + 2\blD_{[i}\Pi_{j]k}\,^{A}.
\end{align}
For $X$, $Y$, $Z$, and $U$ tangent to $N$, it follows from \eqref{gauss} that
\begin{align}\label{sgauss}
\begin{split}
\Om(&\sR(X, Y)T\phi(Z), T\phi(U))  \\
&= \Om( T\phi(\bar{R}(X, Y)Z), T\phi(U)) 
 + \Om(\nabla_{X}\Pi(Y, Z) - \nabla_{Y}\Pi(X, Z), T\phi(U))\\&\quad + \Om(\Pi(X, \bnabla_{Y}Z) - \Pi(Y, \bnabla_XZ) - \Pi([X, Y], Z), T\phi(U))\\
& =  \Om( T\phi(\bar{R}(X, Y)Z), T\phi(U)) - \Om(\Pi(Y, Z), \Pi(X, U)) + \Om(\Pi(X, Z), \Pi(Y, U))\\
&\quad + X\ind(Y, Z, U)  - \ind(Y, Z, \bnabla_{X}U)   - Y\ind(X, Z, U) + \ind(X, Z,\bnabla_{Y}U) \\
&\quad  + \ind(X, \bnabla_{Y}Z, U) - \ind(Y, \bnabla_XZ, U) - \ind(\bnabla_{X}Y - \bnabla_{Y}X, Z, U)\\
& = \phi^{\ast}(\Om)(\bar{R}(X, Y)Z, U) - \Om(\Pi(Y, Z), \Pi(X, U))  + \Om(\Pi(X, Z), \Pi(Y, U)) \\
&\quad + (\bnabla_{X}\ind)(Y, Z, U) - (\bnabla_{Y}\ind)(X, Z, U).
\end{split}
\end{align}
Alternatively, writing $R_{ijAB} = R_{ijA}\,^{C}\Om_{CA}$,
\begin{align}\label{sgauss2}
\sR_{ijAB}\phi_{k}\,^{A}\phi_{l}\,^{B} = \bar{R}_{ijk}\,^{p}\phi_{p}\,^{A}\phi_{l}\,^{B}\Om_{AB} + 2\Pi_{k[i}\,^{A}\Pi_{j]l}\,^{B}\Om_{AB} + 2\bnabla_{[i}\ind_{j]kl} .
\end{align}

\begin{lemma}
Let $(M, \Om)$ be a symplectic manifold and let $\phi: N \to M$ be a smooth symplectic map. Let $\bnabla \in \symcon(N, \phi^{\ast}(\Om))$ be the symplectic connection on $N$ by $\nabla \in \symcon(M, \Om)$. For a symplectic subspace $L \subset T_{p}N$, the symplectic sectional curvatures $\scurv_{p, L}^{N, \bnabla}$ and $\scurv_{\phi(p), T\phi(p)(L)}^{M, \nabla}$ are related by
\begin{align}\label{scurvhereditary}
\scurv_{\phi(p), T\Phi(p)(L)}^{M, \nabla}(T\phi(p)(Z)) = \scurv_{p, L}^{N, \bnabla}(Z) + \frac{2\Om(\Pi(X, Z), \Pi(Y, Z))}{\phi^{\ast}(\Om)(X, Y)}
\end{align}
where $X$ and $Y$ span $L$ and $Z \in L$.
\end{lemma}
\begin{proof}
Since $\bnabla$ is the induced connection, $\ind$ vanishes identically, and the identity \eqref{scurvhereditary} follows from \eqref{sgauss} upon dividing by $\phi^{\ast}(\Om)(X, Y)$.
\end{proof}

Many notions from the pseudo-Riemannian geometry of submanifolds, such as the mean curvature vector or totally umbilic submanifolds, have no obvious analogues in the context of symplectic connections. On the other hand, at the purely tensorial level, potentially interesting conditions on submanifolds can be identified in terms of conditions on the second fundamental form. Such conditions arise most naturally by decomposing by symmetries and traces the second fundamental form and tensors constructed from it. For instance the mean curvature vector is the trace of the second fundamental form with respect to the induced metric, and a submanifold is totally umbilic if its second fundamental form is the tensor product of the induced metric with a vector field. Although, neither condition has an obvious analogue in the symplectic setting, the identity \eqref{scurvhereditary} suggests imposing similar conditions on certain tensors constructed from the induced symplectic form and quadratic in the second fundamental form.

Given a symplectic immersion $N \to (M, \Om)$ and $\nabla \in \symcon(M, \Om)$, let $\om^{ij}$ be the bivector inverse to $\om_{ij} = \phi^{\ast}(\Om)_{ij}$ on $TN$ and define tensors $\smc_{ijkl}$, $\smc_{ij}$, and $\Pi_{ijkl}$ by
\begin{align}\label{smcdefined}
\begin{split}
\Pi_{ijkl} &= 2\Pi_{i(k}\,^{A}\Pi_{l)j}\,^{B}\Om_{AB}, \qquad \smc_{ij} = 
\om^{ab}\Om_{AB}\Pi_{ia}\,^{A}\Pi_{jb}\,^{B},\\
\smc_{ijkl} &= \Pi_{ijkl} - \tfrac{1}{k+1}\left(\om_{i(k}\smc_{l)j} - \om_{j(k}\smc_{l)i} + \om_{ij}\smc_{kl}\right),
\end{split}
\end{align}
where $N$ has dimension $2k$ and  $\Pi$ is the second fundamental form of $\phi$ with respect to $\nabla \in \symcon(M, \Om)$ and the induced symplectic connection $\bnabla \in \symcon(N, \om)$. Then $\smc_{ij}$ is symmetric and $\Pi_{ijkl} = \Pi_{[ij]kl} = \Pi_{ij(kl)}$, $\Pi_{[ijk]l} = 0$, $\om^{pq}\Pi_{pijq} = \smc_{ij}$, and $\om^{pq}\Pi_{pqij} = 2\smc_{ij}$, so $\Pi_{ijkl}$ and $\smc_{ijkl}$ have the symmetries of a symplectic curvature tensor, and $\smc_{ijkl}$ is the completely trace-free part of $\Pi_{ijkl}$. In particular, if $N$ has dimension $2$ then $\smc_{ijkl} = 0$. 

If $X$ and $Y$ span $L$ and $Z \in L$ then $\Om(X, Y)Z = \Om(X, Z)Y - \Om(Y, Z)X$, so \eqref{scurvhereditary} can be rewritten as
\begin{align}\label{scurvhereditary2}
\begin{split}
\scurv_{T\phi(L)}^{M, \nabla}(T\phi(Z)) &= \scurv_{L}^{N, \bnabla}(Z) + 2\Pi_{ijkl}X^{i}Y^{j}Z^{k}Z^{l}\\
& = \scurv_{L}^{N, \bnabla}(Z) + 2\smc_{ijkl}X^{i}Y^{j}Z^{k}Z^{l} + \tfrac{4}{k+1}\smc_{ij}Z^{i}Z^{j}.
\end{split}
\end{align}
(The dependence on the basepoint $p \in M$ is omitted for readability.)

From the point of view of tensor algebra it is natural to consider the conditions $\smc_{ijkl} = 0$ and $\smc_{ij} = 0$ on a symplectic submanifold. In the metric setting there are other considerations that motivate giving attention to conditions such as the vanishing of the mean curvature, namely that this condition characterizes the critical points of the volume functional on submanifolds. While similarly geometric motivations for considering conditions such as $\smc_{ijkl} = 0$ or $\smc_{ij} = 0$ remain to be found, these seem to be among the simplest conditions of geometric origin that can be imposed on a symplectic submanifold of a manifold with symplectic connection. Other conditions that can be considered include requiring $\smc_{ij}$ to be definite or semidefinite.

\begin{lemma}
Let $(M, \Om)$ be a $2n$-dimensional symplectic manifold and let $\phi: N \to M$ be a smooth symplectic immersion of a $2k$-dimensional manifold $N$. Let $\nabla \in \symcon(M, \Om)$ and let $\bnabla \in \symcon(N, \phi^{\ast}(\Om))$ be the induced symplectic connection on $N$. If $\nabla$ has constant symplectic sectional curvature then
\begin{align}\label{cricsub}
(k+1)\phi^{\ast}(R)_{ij}= (n+1)\left(\bar{R}_{ij} + \smc_{ij}\right),
\end{align}
where $\smc_{ij}$ is defined in \eqref{smcdefined} and $\bar{R}_{ij}$ is the Ricci curvature of $\bnabla$. In particular, if $\smc_{ij}$ vanishes then the Ricci curvature of $\bnabla$ is $(k+1)/(n+1)$ times the pullback of the Ricci curvature of $\nabla$.
\end{lemma}
\begin{proof}
By \eqref{symplecticweyl}, that $\nabla$ have constant symplectic sectional curvature means that
\begin{align}\label{indsw}
(n+1)\sR_{ijAB}\phi_{k}\,^{A}\phi_{l}\,^{B} = \phi^{\ast}(\Om)_{i(k}\phi^{\ast}(R)_{l)j} - \phi^{\ast}(\Om)_{j(k}\phi^{\ast}(R)_{l)i} + 2\phi^{\ast}(\Om)_{ij}\phi^{\ast}(R)_{kl}.
\end{align}
Substituting \eqref{indsw} in \eqref{sgauss2} yields
\begin{align}\label{scgauss}
\begin{split}
\phi^{\ast}(\Om)_{i(k}\phi^{\ast}(R)_{l)j}& - \phi^{\ast}(\Om)_{j(k}\phi^{\ast}(R)_{l)i} + 2\phi^{\ast}(\Om)_{ij}\phi^{\ast}(R)_{kl}\\
 &= (n+1)\left(\bar{R}_{ijk}\,^{p}\phi^{\ast}(\Om)_{pl} + 2\Pi_{k[i}\,^{A}\Pi_{j]l}\,^{B}\Om_{AB}\right). 
\end{split}
\end{align}
Contracting \eqref{scgauss} with the bivector $\om^{ij}$ inverse to $\phi^{\ast}(\Om)_{ij}$ yields \eqref{cricsub}.
\end{proof}

By a symplectic affine space is meant a vector space equipped with a flat affine connection and a parallel symplectic form.
\begin{corollary}\label{saffinecorollary}
In particular, the Ricci curvature of the symplectic connection induced on a symplectic submanifold of a symplectic affine space equals $-\smc_{ij}$.
\end{corollary}

Even for two-dimensional submanifolds Corollary \ref{saffinecorollary} does not immediately yield strong conclusions because conditions relating the properties of the Ricci tensor of an affine (or symplectic) connection with the topology of the underlying manifold are not known. For instance, by Proposition $4.1$ of \cite{Kobayashi-ricci}, a two-dimensional torus admits a projectively flat connection that preserves a volume form and has negative definite Ricci tensor.

\section{Nonpositive symplectic sectional curvature}\label{npcsection}
Unlike the situation for Riemannian metrics, there are few results that relate the behavior of the geodesics of a general affine connection to the properties of the curvature. The situation is no better for symplectic connections. The difficulty is that most results about Jacobi fields along Riemannian geodesics use heavily the arc length and its first and second variation, and similar quantities are not immediately available in the affine or symplectic settings. 

Lemma \ref{nonposlemma} shows that if a symplectic connection has nonpositive symplectic sectional curvature then the kernel of the differential $T\exp_{p}(v)$ of the exponential map at $p \in M$ in the direction $v \in T_{p}M$ is contained in the symplectic orthogonal complement $\lb v \ra^{\perp} \subset T_{p}M$ of $v$. That such an apparently weak conclusion requires such an apparently strong hypothesis illustrates the issues in extending metric arguments to the symplectic setting.

Let $(M, \Om)$ be a symplectic manifold. Next it is shown that $\nabla\in \symcon(M, \Om)$ has nonpositive symplectic sectional curvature if and only if there holds the inequality 
\begin{align}\label{npc}
\Om(X, Y)\Om(R(X, Y)X, X) \leq 0
\end{align}
for all vector fields $X$ and $Y$ on $M$. By definition the symplectic sectional curvature is nonpositive if and only if $\Om(X, Y)^{-1}\Om(R(X, Y)X, X) \leq 0$ whenever $\Om(X, Y) \neq 0$. Multiplying both sides by the nonnegative quantity $\Om(X, Y)^{2}$ shows that the symplectic sectional curvature is nonpositive if and only if there holds \eqref{npc} whenever $\Om(X, Y) \neq 0$. On the other hand, \eqref{npc} is always true, trivially, when $\Om(X, Y) = 0$.

Given a symplectic manifold $(M, \Om)$, let $\ga:I \to M$ be a geodesic of $\nabla \in \symcon(M, \Om)$, where $I \subset \rea$ is an open interval. A Jacobi field $J(t)$ along $\ga(I)$ is \textit{coisotropic} if $\Om(J(t), \dot{\ga}(t)) = 0$ for all $t \in I$. 

\begin{lemma}\label{nonposlemma}
Let $(M, \Om)$ be a symplectic manifold and let $\nabla \in \symcon(M, \Om)$ have nonpositive symplectic sectional curvature. Let $\ga:I \to M$ be a maximal geodesic where $I \subset \rea$ is an open interval, and let $J(t)$ be a Jacobi field along $\ga(I)$. Then one of the following mutually exclusive possibilities occurs:
\begin{enumerate}
\item $\Om(J, \dot{\ga})$ has at most one zero in $I$. In particular $J$ has at most one zero on $\ga(I)$.
\item $\Om(J, \dot{\ga})$ is identically zero on $I$.
\end{enumerate}
In particular, if $p, q \in M$ are conjugate with respect to $\ga$ then any Jacobi field along $\ga(I)$ vanishing at $p$ and $q$ is coisotropic. 
\end{lemma}
\begin{proof}
That $J$ be a Jacobi field means $\ddot{J} = R(\dot{\ga}, J)\dot{\ga}$ where $\dot{J} = \nabla_{\dot{\ga}}J$ and $\ddot{J} = \nabla_{\dot{\ga}}\dot{J}$. Hence
\begin{align}\label{convexity}
\begin{split}
\tfrac{d^{2}}{dt^{2}}\left(\tfrac{1}{2}\Om(J, \dot{\ga})^{2}\right)& = \tfrac{d}{dt}\left(\Om(J, \dot{\ga})\Om(\dot{J}, \dot{\ga})\right)\\
& = \Om(\dot{J}, \dot{\ga})^{2} - \Om(\dot{\ga}, J)\Om(R(\dot{\ga}, J)\dot{\ga}, \dot{\ga}) \geq \Om(\dot{J}, \dot{\ga})^{2} \geq 0,
\end{split}
\end{align}
where the first inequality follows from nonpositivity of the symplectic sectional curvature and \eqref{npc}. Consequently, $\Om(J, \dot{\ga})^{2}$ is a nonnegative convex function, so if it vanishes at more than one point it is identically zero. 
\end{proof}

\section{Indefinite symplectic sectional curvature}\label{indefinitesection}
Lemma \ref{sectcharlemma} gives an alternative characterization of the symplectic sectional curvature that will be applied in Lemma \ref{notdefinitelemma}.
\begin{lemma}\label{sectcharlemma}
Let $\nabla \in \symcon(M, \Om)$. If $Z \in \Ga(TM)$, then for any $p \in M$ and any symplectic subspace $L \in \sgr(2, T_{p}M)$ containing $Z_{p}$ there holds
\begin{align}\label{sss}
\scurv_{p, L}(Z)\Om(X, Y) = d((\nabla_{Z}Z)^{\sflat} - \lie_{Z}(Z^{\sflat}))(X, Y) + 2\Om(\nabla_{X}Z, \nabla_{Y}Z).
\end{align}
where $X$ and $Y$ are any vector fields spanning $L$ and $T^{\sflat} = \Om(T, \dum)$ for any vector field $T$.
\end{lemma}
\begin{proof}
Let $X$ and $Y$ be vector fields on $M$. Then
\begin{align}\label{ss1}
\begin{split}
\nabla_{i}(\nabla_{X}Y)_{j} & = X^{p}\nabla_{i}\nabla_{p}Y_{j} + \nabla_{i}X^{p}\nabla_{p}Y_{j}= x^{p}\nabla_{p}\nabla_{i}Y_{j} + R_{pij}\,^{q}X^{p}Y_{q} + \nabla_{i}X^{p}\nabla_{p}Y_{j}.
\end{split}
\end{align}
Skew symmetrizing \eqref{ss1} yields
\begin{align}\label{ss2}
\begin{split}
d((\nabla_{X}Y)^{\sflat})_{ij} & = 2\nabla_{[i}(\nabla_{X}Y)_{j]}  = X^{p}\nabla_{p}(dY^{\sflat})_{ij} + R_{ijpq}X^{p}Y^{q} +  \nabla_{i}X^{p}\nabla_{p}Y_{j} -  \nabla_{j}X^{p}\nabla_{p}Y_{i}\\
& = \lie_{X}(dY^{\sflat})_{ij} + 2 \nabla_{[i}X^{p}\nabla_{j]}Y_{p} + R_{ijpq}X^{p}Y^{q}.
\end{split}
\end{align} 
where the last equality follows from the identity $\lie_{X}\om_{ij} = X^{p}\nabla_{p}\om_{ij} + \om_{pj}\nabla_{i}X^{p} + \om_{ip}\nabla_{j}X^{p}$, valid for any two-form $\om_{ij}$. Rearranging \eqref{ss2} and commuting the exterior and Lie derivatives yields
\begin{align}\label{ss3}
\begin{split}
R_{ijpq}X^{p}Y^{q} = d((\nabla_{X}Y)^{\sflat} - \lie_{X}(Y^{\sflat}))_{ij} + 2\Om_{pq}\nabla_{[i}X^{p}\nabla_{j]}Y^{q}.
\end{split}
\end{align}
Taking $X = Y = Z$ in \eqref{ss3} yields \eqref{sss}.
\end{proof}

\begin{corollary}\label{notdefinitecorollary}
For $\nabla \in \symcon(M, \Om)$, if there is a symplectic vector field $Z$ such that $\nabla_{Z}Z = 0$ then for any $p \in M$ for which $Z_{p} \neq 0$, the symplectic sectional curvature of any symplectic subspace $L \in \sgr(2, T_{p}M)$ containing $Z_{p}$ is not definite.
\end{corollary}
\begin{proof}
That $Z$ be symplectic is equivalent to $\lie_{Z}(Y^{\sflat}) = (\lie_{Z}Y)^{\sflat}$ for any vector field $Y$. In particular $\lie_{Z}(Z^{\sflat}) = 0$. If $L$ contains $Z_{p}$ then $X$ can be taken to be $Z$ in \eqref{sss}, and the right-hand side of \eqref{sss} vanishes, so $\scurv_{p, L}(Z)$.
\end{proof}

A Lie group $G$ equipped with a left-invariant symplectic form $\Om$ is a \textit{symplectic Lie group}. A left-invariant torsion-free connection $\bnabla$ on a Lie group $G$ is determined by a tensor $A_{ij}\,^{k} = A_{ji}\,^{k}$ on the Lie algebra $\g$. If $\lf^{a}$ is the left-invariant vector field generated by $a \in \g$ then $\bnabla_{\lf^{a}}\lf^{b} = \lf^{A(a, b)}$. Applying the construction \eqref{gensym} to the connection $\bnabla$ associated with $A(a, b) = \tfrac{1}{2}[a, b]$ yields a left-invariant symplectic connection $\nabla$ on $G$ for which the corresponding tensor has the form $\sad(a)b$ where $\sad:\g \to \eno(\g)$ is given by
\begin{align}\label{saddefined}
\sad(a) = \tfrac{1}{3}(\ad(a) + \ad(a)^{\ast}).
\end{align}
In \eqref{saddefined} the transformation $\ad(a)^{\ast}$ is the \textit{symplectic adjoint} of $\ad(a)$ defined by
\begin{align}\label{sad2}
\Om(\ad(a)^{\ast}b, c) = -\Om(a, \ad(a)c).
\end{align}
Using \eqref{sad2} it is straightforward to check that $\sad(a)b - \sad(b)a = [a, b]$, so that the connection $\nabla$ associated to \eqref{saddefined} is torsion-free. Similarly, it can be checked directly that $\nabla$ is symplectic. It seems reasonable to call the connection $\nabla$ determined by $\sad$ the \textit{canonical symplectic connection} of the symplectic Lie group $(G, \Om)$.

An element $x \in \g$ of a symplectic Lie algebra $(\g, \Om)$ is \textit{self-adjoint} or \textit{anti-self-adjoint} if $\ad(x)^{\ast} = \ad(x)$ or $\ad(x)^{\ast} = -\ad(x)$. It follows from the definitions that the left-invariant vector field $\lf^{z}$ generated by $z \in \g$ is symplectic if and only if $z$ is self-adjoint.

\begin{lemma}\label{notdefinitelemma}
For the canonical symplectic connection of a connected symplectic Lie group $(G, \Om)$, the symplectic sectional curvature of a symplectic subspace of $\g$ containing a self-adjoint element is not definite.
\end{lemma}

\begin{proof}
For self-adjoint $z \in \g$, $\lf^{z}$ is symplectic and $\sad(z)z = 0$, so this follows from Corollary \ref{notdefinitecorollary}.
\end{proof}

Let $(G, \Om)$ be a symplectic Lie group with lie algebra $\g$. Let $\ell \in \g$ be the element defined by $\Om(\ell, x) = \tr \ad(x)$ for all $x \in \g$. By definition $\ell$ is zero if and only if $\g$ is unimodular. Writing $\Om$ for the nondegenerate two-form on $\g$ generating the given left-invariant two-form $\Om$. By the left-invariance there holds
\begin{align}\label{ellsad}
\Om([\ell, x], y) + \Om(x, [\ell, y]) = \Om(\ell, [x, y]) = \tr \ad([x, y]) = \tr[\ad(x), \ad(y)] = 0.
\end{align}
The identity \eqref{ellsad} shows that $\ad(\ell)^{\ast}= \ad(\ell)$ so $\ell$ is self-adjoint. Lemma \ref{notdefinitelemma} implies:

\begin{corollary}\label{lienotdefinitecorollary}
On a connected symplectic Lie group $(G, \Om)$ that is not unimodular, the symplectic sectional curvature of the canonical symplectic connection of a symplectic subspace of $\g$ containing the element $\ell$ is not definite.
\end{corollary}

Although the proof is not given here, it can be proved that the canonical symplectic connection of a nilpotent symplectic Lie group is Ricci flat. However, as will be apparent from the examples that follow, if the symplectic Lie group is not nilpotent, the Ricci tensor of the canonical connection need not be flat.

\begin{example}
Consider the Lie algebra $\g= \aff(1, \com)$ of affine transformations of the complex line. With respect to the basis $e_{1} = (1, 0)$, $e_{2} = (\j, 0)$, $e_{3} = (0, 1)$, and $e_{4} = (0, \j)$, the Lie bracket is
\begin{align}\label{affbrack}
[x, y] & = (x_{1}y_{3} - x_{3}y_{1} - x_{2}y_{4} + x_{4}y_{2})e_{3} + (x_{1}y_{4} - x_{4}y_{1} + x_{2}y_{3} - x_{3}y_{2})e_{4}.
\end{align}
The commutator is $[\g, \g] = \spn\{e_{3}, e_{4}\}$, which is abelian, but stable under the adjoint action, showing that $\g$ is solvable but not nilpotent. Let $\{e^{1}, e^{2}, e^{3}, e^{4}\}$ be the dual coframe. Straightforward computation using \eqref{affbrack} shows that, for $x = x_{i}e^{i} \in \g^{\ast}$, $d(x_{i}e^{i}) = (x_{3}e^{3} + x_{4}e^{4})\wedge e^{1} + (x_{4}e^{3} - x_{3}e^{4})\wedge e^{2}$. Since $dx \wedge dx = 2(x_{3}^{2} + x_{4}^{2})e^{1}\wedge e^{2}\wedge e^{3} \wedge e^{4}$, the closed two-form $d(x_{i}e^{i})$ is symplectic if and only if $x_{3}^{2} + x_{4}^{2} \neq 0$. Choose $\Om = - de^{4} = e^{1}\wedge e^{4} + e^{2}\wedge e^{3}$. Straightforward computations show
\begin{align}\label{affstar}
\begin{split}
\ad(x)^{\ast} y & = (x_{2}y_{2} - x_{1}y_{1})e_{1} - (x_{1}y_{2} + x_{2}y_{1})e_{2} + (x_{4}y_{2} - x_{3}y_{1})e_{3} - (x_{4}y_{1} + x_{3}y_{2})e_{4}.
\end{split}
\end{align}
Combining \eqref{affbrack} and \eqref{affstar} yields
\begin{align}
\begin{split}
3\sad(x)y & =  (x_{2}y_{2} - x_{1}y_{1})e_{1} - (x_{1}y_{2} + x_{2}y_{1})e_{2} \\&\qquad + (x_{1}y_{3} - 2x_{3}y_{1} - x_{2}y_{4} + 2x_{4}y_{2})e_{3} + (x_{1}y_{4} - 2x_{4}y_{1} + x_{2}y_{3} - 2x_{3}y_{2})e_{4}.
\end{split}
\end{align}
Since $\tr \ad(x) = 2x_{1}$, $\g$ is not unimodular and $\dl = -2e_{4}$. Calculations shows that $[\sad(x), \sad(y)] = \tfrac{2}{3}\sad([x, y])$ so that $R(x, y) = [\sad(x), \sad(y)] - \sad([x, y]) = -\tfrac{1}{3}\sad([x, y]) = -\tfrac{1}{2}[\sad(x), \sad(y)]$. From this and the symplectic self-adjointness of $\sad(x)$ there follows
\begin{align}\label{affcurv}
2\Om(R(x, y)u, v) = -\Om(\sad(x)u, \sad(y)v) + \Om(\sad(y)u, \sad(x)v).
\end{align}
Since $\sad(x)\dl$ is in the isotropic subspace $\spn\{e_{3}, e_{4}\}$, it follows from \eqref{affcurv} that $\Om(R(x, y)\dl, \dl) = 0$ for all $x, y \in \g$. In particular, the symplectic sectional curvature of any symplectic subspace of $\g$ containing $\dl$ is not definite. This conclusion also follows from Corollary \ref{lienotdefinitecorollary}. In fact, in this case the symplectic sectional curvature of any symplectic subspace containing $\dl$ is degenerate but generally nonzero. This follows from 
\begin{align}
\Om(R(x, \dl) u, u) = \tfrac{4}{9}(-x_{1}u_{1}^{2} + x_{1}u_{2}^{2} + 2x_{2}u_{1}u_{2}).
\end{align}
Note that $\Om(x, \dl) \neq 0$ if and only if $x_{1} \neq 0$. If $u = ax + b\dl$ then
\begin{align}
\Om(R(x, \dl), u, u) = \tfrac{4}{9}a^{2}x_{1}(3x_{2}^{2} - x_{1}^{2}),
\end{align}
and $\Om(x, \dl) = -2x_{1}$, so that, for $L = \spn\{x, \dl\}$,
\begin{align}
\scurv_{L}(u) = \tfrac{2}{9}(u_{1}^{2} - 3u_{2}^{2}),
\end{align}
for $u \in L$. This is nonzero as long as $u_{1}^{2} \neq 3u_{2}^{2}$.
Further calculations show $9\ric(x, y) = 4(x_{2}y_{2} - x_{1}y_{1})$. 
Since the Ricci endomorphism $A \in \eno(\g)$ defined by $\Om(Ax, y) = \ric(x, y)$ is given by $9A = 4(e^{1}\tensor e_{4} - e^{2}\tensor e_{3})$ it satisfies $A\circ A = 0$, so is two-step nilpotent. 
Since
\begin{align}
9 \ric(\sad(x)y, z) = 4(-x_{1}y_{1}z_{1} + x_{2}y_{2}z_{1} + x_{2}y_{1}z_{2} + x_{1}y_{2}z_{2}),
\end{align}
and $(\nabla \ric)(x, y, z) = -\ric(\sad(x)y, z) - \ric(y, \sad(x)z)$, there holds
\begin{align}
9x^{i}y^{j}z^{k}\nabla_{(i}R_{jk)} & = 8(x_{1}y_{1}z_{1} - x_{2}y_{2}z_{1} - x_{2}y_{1}z_{2} - x_{1}y_{2}z_{2}),
\end{align}
showing that $\nabla_{(i}R_{jk)}$ is not zero, so that the canonical connection $\nabla$ is not preferred. Consequently, by \eqref{divw}, $\nabla$ is not Weyl flat.
\end{example}

\begin{example}\label{nonexactnonunimodularsection}
Here is given an example of a $4$-dimensional solvable symplectic Lie algebra that is neither unimodular nor exact. This Lie algebra comes from \cite{Ovando} (see Propositions $2.2$ and $2.4$), where it is labeled as $\mathfrak{r}_{4, 0}$. With respect to the basis $\{e_{1}, e_{2}, e_{3}, e_{4}\}$ and dual coframe $\{e^{i}\}$ the Lie bracket is
\begin{align}
\begin{split}
[x, y] & = (x_{4}y_{1} - x_{1}y_{4})e_{1} + (x_{4}y_{3} - x_{3}y_{4})e_{2}.
\end{split}
\end{align}
Let $\{e^{i}\}$ be the coframe dual to $\{e_{i}\}$. Since $\tr \ad(x) = x_{4}$, $\g$ is not unimodular. Since $d(x_{i}e^{i}) = (x_{1}e^{1} + x_{2}e^{3}) \wedge e^{4}$ satisfies $e^{4} \wedge d(x_{i}e^{i}) = 0$, the nondegenerate two-form $\Om = e^{1}\wedge e^{4} + e^{2} \wedge e^{3}$ is closed but not exact. 
Since $\g$ is four-dimensional it must be solvable by a theorem of Chu in \cite{Chu}. This also can be checked directly, for the commutator subalgebra $[\g, \g] =\spn\{e_{1}, e_{2}\}$ is abelian but stable under the adjoint action of $\g$. This also shows that $\g$ is not nilpotent. As
\begin{align}
\begin{split}
\ad(x)^{\ast} y & = -(x_{3}y_{3} + x_{1}y_{4})e_{1} + x_{4}y_{3}e_{2} - x_{4}y_{4}e_{4},
\end{split}
\end{align}
the canonical connection on $\g$ is given by
\begin{align}
\begin{split}
3\sad(x)y & = (x_{4}y_{1} - 2x_{1}y_{4} - x_{3}y_{3})e_{1} + (2x_{4}y_{3} - x_{3}y_{4})e_{2} - x_{4}y_{4}e_{4}.
\end{split}
\end{align}
It can be checked that the curvature tensor $R_{ijkl}$ is represented by
\begin{align}\label{ex2curv}
\tfrac{1}{9}(e^{3} \wedge e^{4})\tensor (e^{3} \tensor e^{4} + e^{4} \tensor e^{3}) - \tfrac{2}{9}(e^{1}\wedge e^{4})\tensor e^{4}\tensor e^{4}.
\end{align}
The element $\dl$ is $\dl = e_{1}$. The subspace $L = \spn\{x, \dl\}$ is symplectic if and only if $x_{4} \neq 0$. By \eqref{ex2curv}, the symplectic sectional curvature of such $L$ is
\begin{align}
\scurv_{L}(u) = -\tfrac{2}{9}u_{4}^{2}.
\end{align}
It follows that the symplectic sectional curvature of any symplectic subspace containing $\dl$ is degenerate. Since $u$ is in the span of $x$ and $\dl = e_{1}$ and $x_{4} \neq 0$, $u_{4}$ is not zero as long as $u$ is not proportional to $\dl$. Hence, in this case, the symplectic sectional curvature of a symplectic subspace containing $\dl$, while degenerate, is never zero.

Using $2R_{ij} = R_{p}\,^{p}\,_{ij}$ and computing $\ric(x, y) = \Om(R(e_{1}, x)y, e_{4}) + \Om(R(e_{2},x)y, e_{3})$ yields that the Ricci tensor of the canonical connection $\nabla$ is $9\ric = -2e^{4}\tensor e^{4}$. (Alternatively, it can be proved that the Ricci curvature of the canonical connection of a solvable symplectic Lie group is given by $9\ric(x, y) = \tr\ad(\ad(x)^{\ast}y) - B(x, y)$, where $B$ is the Killing form; here $B(x, y) = x_{4}y_{4} = -\tr\ad(\ad(x)^{\ast}y)$.)
In particular, $\nabla$ is not Ricci flat. The connection $\nabla$ is not preferred, for 
\begin{align}
\begin{split}
(\nabla \ric)(x, y, z) & = -\ric(\sad(x) y , z) - \ric(y, \sad(x) z)   = -(4/27)x_{4}y_{4}z_{4}.
\end{split}
\end{align}
Since the complete symmetrization of $\nabla \ric$ does not vanish, $\nabla$ is not preferred. Consequently, by \eqref{divw}, $\nabla$ is not Weyl flat. 
\end{example}

\bibliographystyle{amsplain}
\def\cprime{$'$} \def\cprime{$'$} \def\cprime{$'$} \def\cprime{$'$}
  \def\cprime{$'$} \def\cprime{$'$}
  \def\polhk#1{\setbox0=\hbox{#1}{\ooalign{\hidewidth
  \lower1.5ex\hbox{`}\hidewidth\crcr\unhbox0}}} \def\cprime{$'$}
  \def\Dbar{\leavevmode\lower.6ex\hbox to 0pt{\hskip-.23ex \accent"16\hss}D}
  \def\cprime{$'$} \def\cprime{$'$} \def\cprime{$'$} \def\cprime{$'$}
  \def\cprime{$'$} \def\cprime{$'$} \def\cprime{$'$} \def\cprime{$'$}
  \def\cprime{$'$} \def\cprime{$'$} \def\cprime{$'$} \def\dbar{\leavevmode\hbox
  to 0pt{\hskip.2ex \accent"16\hss}d} \def\cprime{$'$} \def\cprime{$'$}
  \def\cprime{$'$} \def\cprime{$'$} \def\cprime{$'$} \def\cprime{$'$}
  \def\cprime{$'$} \def\cprime{$'$} \def\cprime{$'$} \def\cprime{$'$}
  \def\cprime{$'$} \def\cprime{$'$} \def\cprime{$'$} \def\cprime{$'$}
  \def\cprime{$'$} \def\cprime{$'$} \def\cprime{$'$} \def\cprime{$'$}
  \def\cprime{$'$} \def\cprime{$'$} \def\cprime{$'$} \def\cprime{$'$}
  \def\cprime{$'$} \def\cprime{$'$} \def\cprime{$'$} \def\cprime{$'$}
  \def\cprime{$'$} \def\cprime{$'$} \def\cprime{$'$} \def\cprime{$'$}
  \def\cprime{$'$} \def\cprime{$'$} \def\cprime{$'$} \def\cprime{$'$}
  \def\cprime{$'$} \def\cprime{$'$}
\providecommand{\bysame}{\leavevmode\hbox to3em{\hrulefill}\thinspace}
\providecommand{\MR}{\relax\ifhmode\unskip\space\fi MR }
\providecommand{\MRhref}[2]{%
  \href{http://www.ams.org/mathscinet-getitem?mr=#1}{#2}
}
\providecommand{\href}[2]{#2}

\end{document}